\documentclass{article}
\usepackage{amsfonts,amsmath,amsthm,latexsym,amscd,graphicx,xypic,psfrag}
\usepackage{epic, eepic}
\usepackage[T1]{fontenc}
\usepackage{bbold, mathabx}
\usepackage{verbatim}
\usepackage{setspace}
\usepackage[utf8]{inputenc}
\usepackage[colorlinks=true]{hyperref}

\def\<#1,#2>{\langle #1,#2 \rangle}

 \def\bothID{\rlap{\hbox to.97\wd0{\hss\vrule height.06\ht0 width.82\wd0}}
 \copy0\rlap{\kern-.36\wd0\vrule height1.05\ht0 width.05\ht0}\kern.14\wd0}

 \DeclareMathOperator{\spec}{spec}

\DeclareMathOperator{\tr}{tr}

\begin{document}

\title{Four Dimensional Quantum Yang-Mills Theory for Weak Coupling Strength: Mass Gap Implies Quark Confinement}
\author{Simone Farinelli
        \thanks{Simone Farinelli, Aum\"ulistrasse 20,
                CH-8906 Bonstetten, Switzerland, e-mail simone.farinelli@alumni.ethz.ch}
        }
\maketitle

\begin{abstract}
For the quantized Yang-Mills $3+1$ dimensional problem we introduce the Wilson loop, prove an extension of Elitzur's theorem and shown quark confinement for sufficiently small values of the bare coupling constant, provided the existence of the mass gap.\\\\
\vspace{0.2cm}
\vspace{0.2cm}
\end{abstract}

\newtheorem{theorem}{Theorem}[section]
\newtheorem{proposition}[theorem]{Proposition}
\newtheorem{lemma}[theorem]{Lemma}
\newtheorem{corollary}[theorem]{Corollary}
\theoremstyle{definition}
\newtheorem{ex}{Example}[section]
\newtheorem{rem}{Remark}[section]
\newtheorem*{nota}{Notation}
\newtheorem{defi}{Definition}[section]
\newtheorem{conjecture}{Conjecture}[section]
\newtheorem{counterex}{Counterexample}[section]

\tableofcontents

\section{Introduction}
Gauge fields, termed Yang-Mills fields as well, are utilized in particle physics to describe carriers of a fundamental interaction (cf. \cite{EoM02}). As a matter of fact, the electromagnetic field associated to photons in the electrodynamic theory, the field of vector bosons, mediating the weak interaction in the Weinberg-Salam electro-weak theory, and finally, the gluon field, the carrier of the strong interaction, are described
 by means of Yang-Mills fields. The gravitational field can be interpreted as a Yang-Mills field, too (see \cite{DP75}).\par
H. Weyl (1917) was the first one to propose the idea of a connection as a field, which he utilized in his attempt to describe the electromagnetic field with a connection. In 1954, C.N. Yang and R.L. Mills (cf. \cite{MY54}) postulated that the space of intrinsic degrees of freedom of elementary particles
varies with the considered point of the space-time manifold, and that the intrinsic spaces corresponding to different points are not canonically isomorphic.\par
Reformulated in differential geometrical terms, the postulate of Yang and Mills says that the space of intrinsic degrees of freedom of elementary particles is a vector bundle over the space-time manifold, for which no canonical trivialization exists, and the physical fields corresponding to particles are represented by sections of this vector bundle. The differential
evolution equation of a field is described by a connection in the vector bundle, which can be defined as a trivialization of the bundle along the curves in the base, the space-time manifold.
Such a connection with a fixed holonomy group, describing a physical field, is usually called a Yang-Mills field. The equations for a free Yang-Mills field
can be deduced from a variational principle, and turn out to be a natural non-linear generalization of Maxwell's equations (cf.\cite{Br03}).\par
Field theory does not give the complete picture. Since the early part of the 20th century, it has been understood that the description of nature at the subatomic scale requires quantum mechanics, where classical observables correspond to typically non commuting self-adjoint operators on a Hilbert space, and classic notions as ``the trajectory of a particle'' do not apply. Since fields interact with particles, it became clear by the late 1920s that an internally coherent account of nature must incorporate quantum concepts for fields as well as for particles. Under this approach components of fields at different points in space-time become non-commuting operators.\par
The most important Quantum Field Theories describing elementary particle physics are gauge theories formulated in terms of a principal fibre bundle over the Minkowskian space-time with particular choices of the structure group. They are depicted in Table \ref{GT}.\par
\begin{table}[!]
\begin{center}
\begin{tabular}{|l|l|l|}
  \hline
  \textbf{Gauge Theory} & \textbf{Fundamental Forces}& \textbf{Structure Group}\\
  \hline\hline
  Quantum Electrodynamics  & Electromagnetism &  $U(1)$\\
  (QED) & & \\
  \hline
  Electroweak Theory & Electromagnetism  & $\text{SU}(2)\times U(1)$\\
  (Glashow-Salam-Weinberg) & and weak force & \\
  \hline
  Quantum Chromodynamics  & Strong force & $\text{SU}(3)$\\
  (QCD) & and electromagnetism  & \\
  \hline
  Standard Model  & Strong, weak forces    & $\text{SU}(3)\times\text{SU}(2)\times U(1)$\\
   & and electromagnetism & \\
   \hline
   Georgi-Glashow Grand   & Strong, weak forces    & $\text{SU}(5)$\\
   Unified Theory (GUT1)& and electromagnetism & \\
  \hline
   Fritzsch-Georgi-Minkowski    & Strong, weak forces    & $\text{SO}(10)$\\
   Grand Unified Theory (GUT2)& and electromagnetism & \\
  \hline
   Grand Unified  & Strong, weak forces    & $\text{SU}(8)$\\
   Theory (GUT3)& and electromagnetism & \\
  \hline
  Grand Unified  & Strong, weak forces    & $\text{O}(16)$\\
   Theory (GUT4)& and electromagnetism & \\
  \hline
\end{tabular}
\end{center}
\caption{Gauge Theories}\label{GT}
\end{table}
As shown in \cite{JW04}, in order for Quantum Chromodynamics to completely explain the observed world of strong interactions, the theory must
imply:
\begin{itemize}
\item \textbf{Mass gap:} There must exist some positive constant $\eta$ such that the excitation of the vacuum state has energy at least $\eta$. This would explain why the nuclear force is strong but short-ranged, by providing the mathematical evidence that the corresponding exchange particle, the gluon, has non vanishing rest mass.
\item \textbf{Quark confinement:} The physical particle states corresponding to proton, neutron and pion must be $\text{SU}(3)$-invariant. This would explain why individual quarks are never observed.
\item \textbf{Chiral symmetry breaking:} In the limit for vanishing quark-bare masses the vacuum is invariant under a certain subgroup of the full symmetry group acting on the quark fields. This is required in order to account for the ``current algebra'' theory of soft pions.
\end{itemize}

\noindent The Seventh CMI-Millenium prize problem is the following conjecture.
\begin{conjecture}\label{CMI}
For any compact simple Lie group $G$ there exists a nontrivial Yang-Mills theory on the Minkowskian $\mathbf{R}^{1,3}$, whose quantization satisfies Wightman axiomatic properties of Constructive Quantum Field Theory and has a mass gap $\eta>0$.
\end{conjecture}
The conjecture is explained in \cite{JW04}, commented in \cite{Do04} and in \cite{Fa05}, and proved in \cite{Fa24} upon approval by the mathematical physics community. In this paper we will therefore consider it as unproven. Conjecture \ref{CMI} holds for a bare coupling constant $g\in[0,g_0[$. The main contributions of this paper are a generalization of Elitzur's theorem and the proof of quark containment for the quantum Yang-Mills gauge invariant model on the Minkowskian $\mathbf{R}^{1,3}$ provided the mass gap. \par
So far, the most successful quantum field theory in four dimensions has been on the lattice, see \cite{Di13, Di13Bis, Di14} for a very readable overview of Balaban's monumental work, see \cite{Ba84, Ba84Bis, Ba85, Ba85Bis, Ba85Tris, Ba85Quater, Ba87, Ba88, Ba88Bis, Ba89, Ba89Bis}. The next step would be to pass to a double continuum limit of this model, the adiabatic and the ultraviolet limit,  showing that the continuum limit satisfies Osterwalder-Schrader or Wightman's axioms. This program has turned out to be very difficult so far, see \cite{Se82} for an alternative (unfinished) approach.\par
Beside the mass gap problem, the second big question is quark confinement. Quarks are
the constituents of different elementary particles, like protons and neutrons, and
 puzzling are never observed freely in nature. The problem of quark
confinement has been studied extensively by physicists, who do not seem to have found
a satisfactory theoretical explanation accepted by everyone \cite{Gre11}.\par
In lattice gauge theory Wilson \cite{Wi74} showed that quark confinement is equivalent what is now known as Wilson’s area law.
Later, Osterwalder and Seiler \cite{OS78} proved that the area law is always
fulfilled for sufficiently large coupling constant. But, in order for quark confinement to be true,
the area law must hold for all values of the bare coupling constant arbitrarily near to a critical value,
that is for very small ones if this critical value is zero, as it is believed for many theories of interest.
As a counterexample, Guth \cite{Gu80} and Fr\"ohlich and Spencer \cite{FS82} proved that four-dimensional $U(1)$ lattice gauge
theory is not confining for small values of the coupling constant. \par
Showing that the area law holds In lattice gauge theory at weak coupling remains a mainly open problem
in dimension grater than two (\cite{OS78} and \cite{BDI74}). The area law has been shown to hold at weak coupling is the seldom example of the three-dimensional $U(1)$ lattice gauge theory \cite{GM82}. The most interesting QCD case of four-dimensional $SU(3)$
theory is still unsolved.
Other important progresses in the mathematical study of confinement comprise the following works:
Fr\"ohlich \cite{Froe79}, Durhuus and Fr\"ohlich \cite{DF80}, Borgs
and Seiler \cite{BS83}, Brydges and Federbush \cite{BF80} on the connected problem of Debye screening.\par
Finally, Chatterjee gave in his notable paper \cite{Ch21} a rigorous meaning to the unbroken center symmetry condition for a lattice gauge theory to be confining, which has been believed by physicists since
 the work of ’t Hooft \cite{tH78}. Moreover, he proved that if
the center of the gauge group is nontrivial, and correlations decay exponentially under
arbitrary boundary conditions, then center symmetry does not break, and therefore the
theory is confining.\\
We are therefore lead to the following
\begin{conjecture}\label{QConf}
The nontrivial Yang-Mills theory on the Minkowskian $\mathbf{R}^{1,3}$ postulated in Conjecture \ref{CMI} has quark confinement.
\end{conjecture}

\noindent If Conjecture \ref{CMI} holds true for any bare coupling constant $g\in[0,g_0[$, so does Conjecture \ref{QConf}.
This paper is organized as follows. Section 2 presents the classical Yang-Mills equations
and their Hamiltonian formulation for the Minkowskian $\mathbf{R}^{1,3}$.
Section 3 summarizes the main results in \cite{Fa24} where the  quantization of the Yang-Mills Equations is carried out and the axioms of Constructive Quantum Field Theory, the Osterwalder-Schrader's ones and hence the Wightman's ones are verified, and the existence of a positive mass gap proven. The existence of a rigorous the quantum model for $3+1$ D Yang-Mills theory and its mass gap is presented as conjecture. Gauge invariance is treated as well. Section 4 is the main part of this work and introduces the Wilson loop, the cluster theorem, generalizes Elitzur's theorem and proves the quark containement. Section 5 concludes.

\section{Yang-Mills Connections: Classical Theory}\label{YMC}
\subsection{Definitions}\label{ptrel}
A Yang-Mills connection is a connection in a principal fibre bundle over a (pseudo-)Riemannian manifold whose curvature satisfies the harmonicity condition, i.e. the Yang-Mills equation.
\begin{defi}[\textbf{Yang-Mills Connection}]
Let $P$ be a principal $G$-fibre bundle over a pseudoriemannian $m$-dimensional manifold $(M,h)$,
and let $V$ be the complex vector bundle associated with $P$ and $\mathbf{C}^K$, induced by the representation
$\rho:G\rightarrow\text{GL}(\mathbf{C}^K)$, where $K:=\dim(G)$. A connection on the principal fibre bundle $P$ is a Lie-algebra $\mathfrak{g}$ valued $1$-form $\omega$ on $P$,
such that the following properties hold:
\begin{enumerate}
\item[(i)] Let $A\in\mathfrak{g}$ and $A^*$ the vector field on $P$ defined by
\begin{equation}
A_p^*:=\left.\frac{d}{dt}\right|_{t:=0}(p\exp(tA)).
\end{equation}
Then, $\omega(A^*_p)=A$.
\item[(ii)]For $g\in G$ let
\begin{equation}
\begin{split}
&\text{Ad}_g:G\rightarrow G, h\mapsto\text{Ad}_g(h):=L_g\circ R_{g^{-1}}(h)=ghg^{-1}\\
&\text{ad}_g:\mathfrak{g}\rightarrow \mathfrak{g}, A\mapsto\text{Ad}_g(A):=\left.\frac{d}{dt}\right|_{t:=0}(g\exp(tA)g^{-1})\\
\end{split}
\end{equation}
be the adjoint isomorphism and the adjoint representation, respectively.\\
 Then, $R_g^*\omega=\text{ad}_{g^{-1}}\omega$.
\end{enumerate}
The connection $\omega$ on $P$ defines a connection $\nabla$ for the vector bundle $V$, i.e. an operator acting on the space of cross sections of $V$. The vector bundle connection $\nabla$ can be extended to an operator $d:\Gamma(\bigwedge^p(M)\bigotimes V)\rightarrow \Gamma(\bigwedge^{p+1}(M)\bigotimes V)$, by the formula
\begin{equation}
d^{\nabla}(\eta\otimes v):= d\eta\otimes v + (-1)^{p}\eta\otimes \nabla v.
\end{equation}
The operator $\delta^{\nabla}:\Gamma(\bigwedge^{p+1}(M)\bigotimes V)\rightarrow \Gamma(\bigwedge^{p}(M)\bigotimes V)$, defined as the formal adjoint to $d$, is equal to
\begin{equation}
\delta^{\nabla}\eta= (-1)^{p+1}\ast d^{\nabla} \ast,
\end{equation}
\noindent where $\ast$ denotes the Hodge-star operator on the pseudoriemannian manifold $M$.\par
A connection $\omega$ in a principal fibre bundle $P$ is called a Yang-Mills field if the curvature $F:=d\omega+\omega\wedge \omega$,
considered as a $2$-form with values in the Lie algebra $\mathfrak{g}$, satisfies the Yang-Mills equations
\begin{equation}\label{YME}
 \delta^{\nabla}F=0,
\end{equation}
 or, equivalently,
\begin{equation}
 \delta^{\nabla}R^{\nabla}=0,
 \end{equation}
 where $R^{\nabla}(X,Y):=\nabla_X\nabla_Y-\nabla_Y\nabla_X-\nabla_{[X,Y]}$ denotes the curvature of the vector bundle $V$, and is a $2$-form with values in $V$.
\end{defi}
\begin{rem}[\textbf{Local Representations of Connections on Vector and Principle Fibre Bundles}]\label{rem-local}
The local section  $\sigma:U\subset M\rightarrow P$ is defines the local representation of the connection given on the open $U\subset M$ by $A:=\omega\circ \sigma:U\rightarrow\mathfrak{g}$ a Lie-algebra $\mathfrak{g}$ valued
$1$-form on $U$, the fields $A_j(x):=A(x)e_j=\sum_{k=1}^KA_j^k(x)t_k$ define by means of the tangential map $T_e\rho:\mathfrak{g}\rightarrow\mathcal{L}(\mathbf{C}^K)$ of the representation $\rho:G\rightarrow\text{GL}(\mathbf{C}^K)$, with fields of endomorphisms $T_e\rho A_1,\dots,T_e\rho A_m\in\mathcal{L}(V_x)$ for the bundle $V$. Given a basis of the Lie-algebra $\mathfrak{g}$ denoted by $\{t_1,\dots,t_K\}$, the endomorphisms $\{w_s:=T_e\rho.t_s\}_{s=1,\dots,K}$ in $\mathcal{L}(\mathbf{C}^K)$ have matrix representations with respect to a local basis $\{v_s(x)\}_{s=1,\dots,K}$ denoted by $[w_s]_{\{v_s(x)\}}$. Since $\rho$ is a representation, $T_e\rho$ has maximal rank and the endomorphisms are linearly independent. Given a local basis $\{e_j(x)\}_{j=1,\dots,m}$ for $x\in U\subset M$, the Christoffel symbols of the connection $\nabla$ are locally defined by the equation
\begin{equation}
\nabla_{e_j}v_s=\sum_{r=1}^K\Gamma_{j,s}^rv_r,
\end{equation}
holding true on $U$, and they ,satisfy the equalities
\begin{equation}
\Gamma_{j,s}^r=\sum_{a=1}^K[w_a]_s^rA^a_j.
\end{equation}
Given a local vector field $v=\sum_{s=1}^Kf^sv_s$ in $V|_U$ and a local vector field $e$ in $TM|_U$, the connection $\nabla$ has a local representation
\begin{equation}
\nabla_ev=\sum_{s=1}^K(df^s(e).v_s+f^s\omega(e).v_s),
\end{equation}
where $\omega$ is an element of $T^*U|_U\bigotimes \mathcal{L}(V|_U)$, i.e. an endomorphism valued $1-$form satisfying
\begin{equation}
\omega(e_j)v_s=\sum_{r=1}^K\Gamma_{j,s}^rv_r.
\end{equation}
\end{rem}
\begin{rem}The curvature $2$-form reads in local coordinates as
\begin{equation}
F=\sum_{1\le i<j\le M}\sum_{k=1}^KF_{i,j}^k\,t_k\,dx_i\wedge dx_j=\frac{1}{2}\sum_{i,j=1}^M\sum_{k=1}^K\left(\partial_j A^k_i-\partial_i A_j^k-\sum_{a,b=1}^KC^k_{a,b}A^a_iA_j^b\right)t_k\,dx_i\wedge dx_j,
\end{equation}
where $C=[C^c_{a,b}]_{a,b,c=1,\dots,K}$ are the \textit{structure constants} of the Lie-algebra $\mathfrak{g}$ corresponding to the basis $\{t_1,\dots,t_K\}$, which means that for any $a,b$
\begin{equation}
[t_a,t_b]=\sum_{c=1}^KC_{a,b}^ct_c.
\end{equation}
\end{rem}

\subsection{Hamiltonian Formulation for the Minkowski Space}
 The Hamilton function describes the dynamics of a physical system in classical mechanics by means of Hamilton's equations. Therefore, we have to reformulate the Yang-Mills equations in Hamiltonian mechanical terms. We focus our attention on the Minskowski $\mathbf{R}^4$ with the pseudoriemannian structure of special relativity $h=dx^0\otimes dx^0-dx^1\otimes dx^1-dx^2\otimes dx^2-dx^3\otimes dx^3$. The coordinate $x^0$ represents the time $t$, while $x^1,x^2,x^3$ are the space coordinates.\par We introduce Einstein's summation notation, and adopt the convention that indices for coordinate variables from the greek alphabet vary over $\{0,1,2,3\}$, and those from the latin alphabet vary over the space indices $\{1,2,3\}$. For a generic field $F=[F_\alpha]_{\alpha=0,1,2,3}$ let $\mathbf{F}:=[F_i]_{i=1,2,3}$ denote the ``space'' component. The color indices lie in $\{1,\dots,K\}$. Let \begin{equation}\varepsilon^{a,b,c}:=\left\{
                                                       \begin{array}{ll}
                                                         +1 & \hbox{($\pi$ is even)} \\
                                                         -1 & \hbox{($\pi$ is odd)} \\
                                                         \;\;\; 0 & \hbox{(two indices are equal),}
                                                       \end{array}
                                                     \right.
\end{equation}and any other choice of lower and upper indices, be the Levi-Civita symbol, defined by mean of the permutation $\pi:=\left(
                              \begin{array}{ccc}
                                1 & 2 & 3 \\
                                a & b & c \\
                              \end{array}
                            \right)$ in $\mathfrak{S}^3$.
\begin{rem}
If the Lie-group $G$ is simple, then the Lie-Algebra is simple, and the structure constants can be written as
\begin{equation}
C_{a,b}^c = g\varepsilon^c_{a,b},
\end{equation}
for a positive constant $g$ called \textit{(bare) coupling constant}, (see f.i. \cite{We05} Chapter 15, Appendix A).
  The components of the curvature then read
\begin{equation}
F_{\mu,\nu}^k=\frac{1}{2}(\partial_\nu A^k_\mu-\partial_\mu A_\nu^k-g\varepsilon^k_{a,b}A^a_\mu A^b_\nu).
\end{equation}
We will consider only simple Lie groups. As we will see, it is essential for the existence of a mass gap for the group $G$ to be non-abelian.\\
\end{rem}
\noindent We need to introduce an appropriate gauge for the connections we are considering.
\begin{defi}[\textbf{Coulomb Gauge}]
A connection $A$ over the Minkowski space satisfies the \textit{Coulomb gauge} if and only if
\begin{equation}
A_0^a=0\quad\text{and}\quad\partial_jA_j^a=0
\end{equation}
for all $a=1,\dots,K$ and $j=1,2,3$.
\end{defi}
\begin{defi}[\textbf{Transverse Projector}]
Let $\mathcal{F}$ be the Fourier transform on functions in $L^2(\mathbf{R}^3,\mathbf{R})$. The transverse projector $T:L^2(\mathbf{R}^3,\mathbf{R}^3)\rightarrow L^2(\mathbf{R}^3,\mathbf{R}^3)$ is defined as
\begin{equation}
(Tv)_i:=\mathcal{F}^{-1}\left(\left[\delta_{i,j}-\frac{p_ip_j}{|p|^2}\right]\mathcal{F}(v_j)\right),
\end{equation}
and the vector field $v$ decomposes into a sum of a \textit{transversal} ($v^{\perp}$) and a \textit{longitudinal} ($v^{\parallel}$) component:
\begin{equation}
v_i=v_i^{\perp}+v_i^{\parallel},\quad v_i^{\perp}:=(Tv)_i,\quad v_i^{\parallel}:= v_i-(Tv)_i.
\end{equation}
\end{defi}
\begin{rem}
The Coulomb gauge condition for the space part of a connection $A$ is equivalent to the vanishing of its longitudinal component:
\begin{equation}
{A_i^a}^{\parallel}(t,\cdot)=0
\end{equation}
for all $i=1,2,3$, all $a=1,\dots K$ and any $t\in\mathbf{R}$. The time part $A_0$ of the connection $A$ vanishes by definition of Coulomb gauge.
\end{rem}
\begin{proposition}\label{ModGreen}
For a simple Lie-group as structure group let $A$ be a connection over the Minkowskian $\mathbf{R}^4$ satisfying the Coulomb gauge, and assume that $A_i^a(t,\cdot)\in C^{\infty}(\mathbf{R}^3,\mathbf{R})\cap L^2(\mathbf{R}^3,\mathbf{R})$ for all $i=1,2,3$, all $a=1,\dots K$ and any $t\in\mathbf{R}$. The operator $L$ on the real Hilbert space $L^2(\mathbf{R}^3,\mathbf{R}^K)$ defined as
\begin{equation}
L=L(\mathbf{A};x)=[L^{a,b}(\mathbf{A};x)]:=[\delta^{a,b}\Delta_x^{\mathbf{R}^3}+g\varepsilon^{a,c,b}A_k^c(t,x)\partial_k]
\end{equation}
is essentially self adjoint and elliptic for any time parameter $t\in\mathbf{R}$. Its spectrum lies on the real line, and decomposes into discrete $\spec_d(L)$ and continuous spectrum $\spec_c(L)$.  If $0$ is an eigenvalue, then it has finite multiplicity, i.e. $\ker(L)$ is always finite dimensional .\par
The modified Green's function $G=G(\mathbf{A};x,y)=[G^{a,b}(\mathbf{A};x,y)]\in\mathcal{S}^{\prime}(\mathbf{R}^3,\mathbf{R}^{K\times K})$ for the operator $L$ is the distributional solution to the equation
\begin{equation}\label{L}
L^{a,b}(\mathbf{A};x)G^{b,d}(\mathbf{A};x,y)=\delta^{a,d}\delta(x-y)-\sum_{n=1}^{N}\psi_n^a(\mathbf{A};x)\psi_n^d(\mathbf{A};y),
\end{equation}
where $\{\psi_n(\mathbf{A}; \cdot)\}_n$ is an o.n. $L^2$-basis of $N$-dimensional $\ker(L)$. In equation (\ref{L}) $x$ is seen as variable, while $y$ is considered as a parameter.
This modified Green's function can be written as a Riemann-Stielties integral: For any $\varphi\in\mathcal{S}(\mathbf{R}^3,\mathbf{R}^K)\cap L^2(\mathbf{R}^3,\mathbf{R}^K)$
\begin{equation}\label{gen}
G(\mathbf{A};x,\cdot)(\varphi)=\int_{\lambda\neq0}\frac{1}{\lambda}d(E_{\lambda}\varphi)(x),
\end{equation}
where $(E_{\lambda})_{\lambda\in\mathbf{R}}$ is the resolution of the identity corresponding to $L$.
\end{proposition}
\begin{rem}
In \cite{Br03} and \cite{Pe78} the modified Green's function is constructed assuming that the operator $L$ has a discrete spectral resolution $(\psi_n(\mathbf{A};\cdot),\lambda_n)_{n\ge0}$ as
\begin{equation}\label{part}
G(\mathbf{A};x,y)=\sum_{n:\lambda_n\neq0}\frac{1}{\lambda_n}\psi_n(\mathbf{A};x)\psi_n^{\dagger}(\mathbf{A};y).
\end{equation}
In particular, we have the symmetry property
 \begin{equation}
G(\mathbf{A};x,y)^{\dagger}=G(\mathbf{A};y,x)
\end{equation}
\noindent for all $x,y,\mathbf{A}$ for which the expression is well defined. Since the discontinuity points of the spectral resolution $(E_{\lambda})_{\lambda\in\mathbf{R}}$ are the eigenvalues, i. e. the elements of $\spec_d(L)$ (cf. \cite{Ri85}, Chapter 9), the solution (\ref{gen}) extends (\ref{part}) to the general case.
\end{rem}

\noindent After this preparation we can turn to the Hamiltonian formulation of Yang-Mills' equations, following the results
in \cite{Br03} and \cite{Pe78}.
\begin{theorem}\label{CHam}. For a simple Lie group as structure group and for canonical variables  satisfying the Coulomb gauge condition, the Yang-Mills equations for the Minkowskian $\mathbf{R}^4$  can be written as Hamilton equations
\begin{equation}
\left\{
  \begin{array}{ll}
    \frac{d\mathbf{E}}{dt}&=-\frac{\partial H}{\partial \mathbf{A}}(\mathbf{A},\mathbf{E})\\
    \\
    \frac{d\mathbf{A}}{dt}&=+\frac{\partial H}{\partial \mathbf{E}}(\mathbf{A},\mathbf{E})
  \end{array}
\right.
\end{equation}
for the following choices:
\begin{itemize}
\item \textbf{Position variable:} $\mathbf{A}=[A_{i}^a(t,x)]_{\substack{ a=1,\dots,K \\ i=1,2,3}}$ also termed {\it potentials} ,
\item \textbf{Momentum variable:} $\mathbf{E}=[E_{i}^a(t,x)]_{\substack{ a=1,\dots,K \\ i=1,2,3}}$, whose entries are termed {\it chromoelectric fields},
\item \textbf{Hamilton function:} defined as a function of $\mathbf{A}$ and $\mathbf{E}$ as
 \begin{equation}\label{classicH}
  H=H(\mathbf{A},\mathbf{E}):=\frac{1}{2}\int_{\mathbf{R}^3}d^3x\left(E_i^a(t,x)^2+B_i^a(t,x)^2+f^a(t,x))\Delta f^a(t,x)+2\rho^c(t,x)A^c_0(t,x)\right),
 \end{equation}
 \noindent where $\mathbf{B}=[B_i^a]$, whose entries are termed {\it chromomagnetic fields}, is the matrix valued-function defined as
 \begin{equation}
     \begin{split}
       B_i^a&:=\frac{1}{4}\varepsilon_{i}^{j,k}\left(\partial_jA_k^a-\partial_kA^a_j+g\varepsilon^a_{b,c}A_j^bA_k^c\right),
     \end{split}
 \end{equation}
\noindent and $\rho=[\rho^a(t,x)]$, termed {\it charge density}, is  the vector valued function defined as
\begin{equation}\label{cd}
\rho^a:=g\varepsilon^{a,b,c}E^b_iA^c_i,
\end{equation}
\noindent and where
\begin{equation}\label{deff}
\begin{split}
f^a(t,x)&:=-\int_{\mathbf{R}^3}d^3y\,G^{a,b}(\mathbf{A};x,y)\rho^b(t,y)=-G^{a,b}(\mathbf{A};x,\cdot)(\rho^b(t,\cdot))\\
A_0^a(t,x)&:=\int_{\mathbf{R}^3}d^3y\,G^{a,b}(\mathbf{A};x,y)\Delta f^b(t,y)=G^{a,b}(\mathbf{A};x,\cdot)(\Delta f^b(t,\cdot))\\
\end{split}
\end{equation}
\noindent for the modified Green's function $G(\mathbf{A};x,y)$ for the operator $L(\mathbf{A};x)$.
\end{itemize}
We consider position $\mathbf{A}$ and momentum variable $\mathbf{E}$ as elements of $\mathcal{S}(\mathbf{R}^3,\mathbf{C}^{K \times 3})$ depending on the time parameter $t$, so that the  RHSs of equations (\ref{deff}) are well defined distributions applied to test functions.
\end{theorem}
\begin{rem}
As shown in \cite{Pe78} the ambiguities discussed by Gribov in \cite{Gr78} concerning the gauge fixing (see also \cite{Si78} and \cite{He97}) can be traced precisely to the existence of zero eigenfunctions of the operator $L$.
\end{rem}

\begin{corollary}\label{corClassicH}
The Hamilton function (\ref{classicH}) for the Yang-Mills equations can be written as
\begin{equation}
H=H_I+H_{II}+V,
\end{equation}
where
\begin{equation}\label{classicHamiltonFunction}
\begin{split}
H_I&=\frac{1}{2}\int_{\mathbf{R}^3}d^3x\,E_i^a(t,x)^2\\
& \\
H_{II}&=\frac{g^2}{2}\int_{\mathbf{R}^3}d^3x\left[\int_{\mathbf{R}^3}d^3y\,\partial_iG^{a,b}(\mathbf{A};x,y)\varepsilon^{b,c,d}A_k^d(t,y)E_k^c(t,y)\right]^2\\
& \\
V&=\frac{1}{16}\int_{\mathbf{R}^3}d^3x\,\varepsilon_i^{j,k}\varepsilon_i^{p,q}[(\partial_jA^a_k-\partial_kA_j^a+g\varepsilon^{a,b,c}A_j^bA_k^c)(\partial_pA^a_q-\partial_qA_p^a+g\varepsilon^{a,b,c}A_p^bA_q^c)](t,x).
\end{split}
\end{equation}
\end{corollary}

\section{Quantization of Yang-Mills Equations and Positive Mass Gap}\label{QYM}
There are several methods of designing a quantum theory for non-Abelian gauge fields. The Hamiltonian formulation is the approach used in the original work by Yang-Mills (\cite{MY54}), which was later abandoned in favour of an alternative method based on Feynman path integrals (\cite{FP67}). When it became clear that the Faddeev-Popov method must be incomplete beyond perturbation theory, Hamiltonian formulation enjoyed a partial renaissance. More recent, didactically accessible, examples of the Hamiltonian approach in the physical literature can be found in \cite{Sch08}.\par
In the first section of this chapter we summarize the construction a quantized Yang-Mills theory in dimension $3+1$ carried out in \cite{Fa24}, in the second we treat invariance, and in the third we show the mass gap, the spectral lower bound for the Hamilton operator.


\subsection{Quantization}\label{quant}
In the Yang-Mills $3+1$ dimensional set up,  in order to account for functionals on transversal fields as required by the Coulomb gauge, we introduce the configuration space $\mathcal{S}^{\prime}_{\bot}(\mathbf{R}^4,\mathbf{C}^{K \times 3})$ and the path space $\mathcal{S}^{\prime}_{\bot}(\mathbf{R}^3,\mathbf{C}^{K \times 3})$. These are the duals in the sense of nuclear spaces of the test functions satisfying the transversal condition:
\begin{equation}
\begin{split}
L^2_{\bot}(\mathbf{R}^3,\mathbf{C}^{K \times 3},d^3x)&:=\{\left.\mathbf{A}\in L^2(\mathbf{R}^3,\mathbf{C}^{K \times 3},d^3x)\right|\,\mathbf{A}^{\parallel}=0\},\\
\mathcal{S}_{\bot}(\mathbf{R}^3,\mathbf{C}^{K \times 3})&:=\mathcal{S}(\mathbf{R}^3,\mathbf{C}^{K \times 3})\cap L^2_{\bot}(\mathbf{R}^3,\mathbf{C}^{K \times 3},d^3x)\\
\mathcal{S}_{\bot}(\mathbf{R}^4,\mathbf{C}^{K \times 3})&:=\{f\in\mathcal{S}(\mathbf{R}^4,\mathbf{C}^{K \times 3})\left|\;f(t,\cdot)\in L^2_{\bot}(\mathbf{R}^3,\mathbf{C}^{K \times 3},d^3x)\text{ for all }t\in\mathbf{R}\right.\}
\end{split}
\end{equation}
Note that we do not need to bother about the time component of the connection, because it vanishes in the Coulomb gauge. The tempered distributions are defined
\begin{equation}
\begin{split}
\mathcal{S}_{\bot}^{\prime}(\mathbf{R}^3,\mathbf{C}^{K \times 3})&:=\left\{\mathbf{A}:\mathcal{S}_{\bot}(\mathbf{R}^3,\mathbf{C}^{K \times 3})\rightarrow \mathbf{C}^{K \times 3}\text{ linear and continuous}\right\}\\
\mathcal{S}_{\bot}^{\prime}(\mathbf{R}^4,\mathbf{C}^{K \times 3})&:=\left\{\mathbf{A}:\mathcal{S}_{\bot}(\mathbf{R}^3,\mathbf{C}^{K \times 4})\rightarrow \mathbf{C}^{K \times 3}\text{ linear and continuous}\right\},
\end{split}
\end{equation}
and $\mathbf{A}\in\mathcal{S}_{\bot}^{\prime}(\mathbf{R}^4,\mathbf{C}^{K \times 3})$ is called \textbf{regular}
if there exists $\mathbf{a}=\mathbf{a}(t,x)\in L^2_{\text{loc}}(\mathbf{R}^4,\mathbf{C}^{K\times 3})$ such that
for all $\varphi\in\mathcal{S}_{\bot}(\mathbf{R}^4,\mathbf{C}^{K \times 3})$
\begin{equation}
\mathbf{A}(\varphi)=\int_{\mathbf{R}^4}d^4(t,x)a(t,x).\varphi(t,x),
\end{equation}
where the dot denotes the pointwise multiplication.\par
We define the Hilbert space $\mathcal{E}:=L^2(\mathcal{S}^{\prime}_{\bot}(\mathbf{R}^4,\mathbf{C}^{K \times 3}), d\mu)$
and the physical Hilbert space as $\mathcal{H}:=L^2(\mathcal{S}^{\prime}_{\bot}(\mathbf{R}^3,\mathbf{C}^{K \times 3}), d\nu)$
for appropriate probability measures $\mu$ and $\nu$.\\
Following \cite{Fa24} we introduce the Hamilton operator originated by the quantization of the Hamiltonian formulation of Yang-Mills equations.
This operator is the infinitesimal generator of a time inhomogeneous It\^{o}'s diffusion, whose probability density solves the heat kernel equation. A probability measure on the tempered distributions is constructed by means of this It\^{o}'s process as integrator and utilizing the Feynman-Ka\v{c} formula.
The Osterwalder-Schrader axioms hold true such that the Hamilton operator is selfadjoint on the probability space of the time zero
 tempered distributions. Equivalently, the Wightman axioms are satisfied.\par

\begin{conjecture}[\textbf{4D-YM-Measure Properties}]\label{cYMM}
The classic Hamiltonian system described in Corollary \ref{classicH} can be quantized as follows.\\
A probability measure $\mu^g$ on the space of generalized fields satysfying the Coulomb gauge $\mathcal{S}_{\bot}^{\prime}(\mathbf{R}^4,\mathbf{C}^{K \times 3})$ can be constructed in such manner that there exists a $g_0\in[0,1[$, so that for any bare coupling constant $g\in[0,g_0[$
the generating functional
\begin{equation}\label{es_lambda}
S^{g}(\varphi)=\int_{\mathcal{S}_{\bot}^{\prime}(\mathbf{R}^4,\mathbf{C}^{K \times 3})}e^{\imath\mathbf{A}(\varphi)}d\mu^{g}(\mathbf{A}),
\end{equation}
\noindent for $\varphi\in\mathcal{S}_{\bot}(\mathbf{R}^4,\mathbf{C}^{K \times 3})$ satisfies
the Osterwalder-Schadrer axioms (OS0)-(OS4) and hence the Wightman axioms (W1)-(W8) as in \cite{GJ87}. Note that $S^{g}(\varphi)$ and $A(\varphi)$ are $K\times 3$ complex matrices, and that the exponential is meant componentwise.\par The operator $H^{g}$, defined by means of the reconstruction theorem of quantum mechanics, is selfadjoint, has a domain in the Hilbert space
$L^2(\mathcal{S}_{\bot}^{\prime}(\mathbf{R}^3,\mathbf{C}^{K \times 3}), d\nu^{g})$, has
fundamental state $\Omega_0^{g}=\Omega_0^{g}(\mathbf{A}(s,x))$, i.e. $H^{g}\Omega_0^{g}=0$, and satisfies the Feynman-Ka\v{c}-Nelson formula
\begin{equation}\label{repthm2}
\left(
\Omega_0^{g},B_1e^{-(s_2-s_1)H^{g}}B_2e^{-(s_3-s_2)H^{g}}\cdot\dots\cdot
B_N\Omega_0^{g}
\right)^{g}=
\int_{\mathcal{S}_{\bot}^{\prime}(\mathbf{R}^4,\mathbf{C}^{K \times 3})}\hspace{-0.5cm}\Pi_{k=1}^NB_k(\mathbf{A}(s_k,\cdot))d\mu^{g}(\mathbf{A}),
\end{equation}
\noindent where the scalar product $\left(\cdot,\cdot\right)^{g}$ on  $\mathcal{S}_{\bot}^{\prime}(\mathbf{R}^3,\mathbf{C}^{K \times 3})$ is defined as
\begin{equation}
\left(\Upsilon,\Theta\right)^{g}:=\int_{\mathcal{S}_{\bot}^{\prime}(\mathbf{R}^3,\mathbf{C}^{K \times 3})}\;\Upsilon(\mathbf{A})\bar{\Theta}(\mathbf{A})d\nu^{g}(\mathbf{A}),
\end{equation}
the functionals $(B_k=B_k(\mathbf{A}))_{k=1,\dots,N}$ are in
 $L^2(\mathcal{S}_{\bot}^{\prime}(\mathbf{R}^3,\mathbf{C}^{K \times 3}), \mathbf{C}, d\nu^{g})$,
 and $(s_k)_{k=1,\dots,N}$ is a partition of the interval $[\tau,T]$ defined as $s_k:=\tau+k\frac{T-\tau}{N}$
 for $\tau:=-\frac{t}{2}$ and $T:=+\frac{t}{2}$.\par

The Hamiltonian $H^{g}$ is a  non-negative operator $H^g$ on
$L^2(\mathcal{S}^{\prime}_{\bot}(\mathbf{R}^3,\mathbf{C}^{K \times 3}),\mathbf{C},d\nu^g)$.
If the coupling constant $g$ vanishes, both measures $\mu^0$ and $\nu^0$ are Gaussian, otherwise not.
The domain of definition is
\begin{equation}
\mathcal{D}(H^g):=\left\{\Psi\in L^2(\mathcal{S}^{\prime}_{\bot}(\mathbf{R}^3,\mathbf{R}^{K\times3}),\mathbf{C},d\nu^g)\left|\,H\Psi\in L^2(\mathcal{S}^{\prime}_{\bot}(\mathbf{R}^3,\mathbf{R}^{K\times3}),\mathbf{C},d\nu^g)\right.\right\}.
\end{equation}
Moreover, the operator $H^g$ can be decomposed on $\mathcal{D}(H^g)\cap L^2(L^2_{\bot}(\mathbf{R}^3,\mathbf{R}^{K\times3}),\mathbf{C},d\nu^g)$ as
\begin{equation}
H^g=H_I+H_{II}^g+V^g-V^g_0,
\end{equation}
where
\begin{equation}\label{formH}
\begin{split}
H_{I}&=-\frac{1}{2}\int_{\mathbf{R}^3}d^3x\left[\frac{\delta}{\delta A_i^a(t,x)}\right]^2\\
&\\
H_{II}^g&=-\frac{g^2}{2}\int_{\mathbf{R}^3}d^3x\left[\int_{\mathbf{R}^3}d^3y\,\partial_iG^{a,b}(\mathbf{A}(t,y);x,y)\varepsilon^{b,c,d}A_k^d(t,y)\frac{\delta}{\delta A_k^c(t,y)}\right]^2\\
&\\
V^g&=\int_{\mathbf{R}^3}d^3 x\,|R^{\nabla^\mathbf{A}}(t,x)|^2
\end{split}
\end{equation}
for $\mathbf{A}\in L^2_{\bot}(\mathbf{R}^3,\mathbf{C}^{K \times 3},d^3x)$, where $V^{g}_0$ is a real constant which can be chosen so that
the ground state $\Omega_0^g$ satisfies
\begin{equation}\label{gswl}
H^g\Omega_0^g=0.
\end{equation}
\end{conjecture}

\noindent A proof of Conjecture \ref{cYMM} can be found in \cite{Fa24}.

\subsection{Gauge Invariance}
The construction of the Hamiltonian in Subsection \ref{quant} is gauge invariant. As a matter of fact,
if we repeat the construction for a principal fibre bundle subject to a gauge transformation preserving the Coulomb gauge, we obtain an Hamiltonian which is unitary equivalent with the
original one and has, in particular, the same spectrum.
\begin{defi}[\textbf{Gauge Transformation}] Let $P$ be a principal fibre bundle over a manifold $M$ and $\pi:P\rightarrow M$ be the projection.
An automorphism of $P$ is a diffeomorphism $f:P\rightarrow P$ such that $f(ph)=f(p)h$ for all $h\in G$, $p\in P$. A \textbf{gauge transformation} of $P$
is an automorphism $f:P\rightarrow P$ such that $\pi(p)=\pi(f(p))$ for all $p\in P$. In other words $f$ induces a well defined diffeomorphism
$\bar{f}:M\rightarrow M$ given by $\bar{f}(\pi(p))=\pi(f(p))$.
\end{defi}
Following section 3.3 of \cite{Bl05} we notice that the Lagrangian density on the principal fibre bundle $P$ on which we define the Yang-Mills connection is
a $G$-invariant functional on the space of $1$-jets of maps from $P$ to the fibre of the vector bundle $V$ associated with $P$ induced by the representation
$\rho:G\rightarrow\text{GL}(\mathbf{C}^{3K})$. Hence, the position variable $\mathbf{A}$ occurring in the Lagrangian density and its Legendre transform, the
Hamiltonian density takes value in $\mathbf{C}^{3K}$, which is the fibre of the complex vector bundle $V$. We want to analyze how the position variable behaves
if the principal fibre bundle is subject to a gauge transformation.
\begin{proposition}
Let $f$ be a gauge transformation of the principal fiber bundle $P$ and $\omega$ a connection. Then, $\omega^f:=(f^{-1})^*\omega$ is a connection on $P$.
They have the local representation on $\pi^{-1}(U)$
\begin{equation}
\begin{split}
\omega_p&=\text{ad}_{\zeta(p)^{-1}}\circ\pi^*A+\zeta^*\theta\\
\omega_p^f&=\text{ad}_{\zeta(p)^{-1}}\circ\pi^*A^f+\zeta^*\theta,
\end{split}
\end{equation}
and
\begin{equation}\label{A-trans}
A^f=\text{ad}_{\phi}\circ(A-\phi^*\theta),
\end{equation}
where:
\begin{itemize}
\item $\pi:P\rightarrow M$ is the projection of the principal fibre bundle $P$ onto its base space $M$,
\item $U\subset M$ is an open subset of the base space,
\item $\psi:\pi^{-1}(U)\rightarrow U\times G$ is a local trivialization of $\pi^{-1}(U)\subset P$, that is a $G$-equivariant diffeomorphism
such that the following diagram commutes
\begin{equation}
\xymatrix{
   \pi^{-1}(U)\ar[d]^{\pi} \ar[r]^{\psi}& U\times G \ar[ld]^{\text{pr}_1}\\
    U & \\
                 }
\end{equation}
This means that $\psi(p)=(\pi(p),\zeta(p))$, where $\zeta:\pi^{-1}(U)\rightarrow G$ is a fibrewise diffeomorphism satisfying
$\zeta(ph)=\zeta(p)h$ for all $h\in G$.
\item the trivialization map $\psi(f(p))=(\pi(p),\zeta(f(p)))$ let us define $\bar{\phi}:\pi^{-1}(U)\rightarrow G$ by
$\bar{\phi}(p):=\zeta(f(p))\zeta(p)^{-1}$, whence $\bar{\phi}(p)=\phi(\pi(p))$ for a well defined function $\phi$ in virtue of the equivariance of
$\psi$ and $f$.
\item The Maurer-Cartan form is the $\mathfrak{g}$-valued $1$-form defined by $\theta_h:=(L_{h^{-1}})_*:T_hG\rightarrow T_eG=\mathfrak{g}$.
\item $A$ and $A^f$ are $\mathfrak{g}$-valued $1$-forms on $M$ introduced in Remark \ref{rem-local}.
\end{itemize}
\end{proposition}
\begin{proof}
These are collected results from Proposition 3.3 and Proposition 3.22 in \cite{Bau14}.\\
\end{proof}
\begin{rem} For matrix groups equation (\ref{A-trans}) becomes
\begin{equation}
A^f=\phi A \phi^{-1}-d\phi \phi^{-1}
\end{equation}
\end{rem}
For the Yang-Mills construction we denote the $K\times 3$ matrices of the local representation of $A$ and its gauge transformation $A^f$ by $\mathbf{A}$
and $\mathbf{A}^f$, which is in line with the notation utilized so far for the position variable and introduced in Theorem \ref{CHam}.

\begin{theorem}\label{uequiv}
Let $f$ be a gauge transform preserving the Coulomb gauge for the Yang-Mills construction and let $H^g$ and $H^{g;f}$  be the Hamilton operators for the quantized Yang-Mills equation with coupling constant $g$, before and after the gauge transform, as shown in Conjecture \ref{cYMM}.
Let $U$ be the operator on $L^2(\mathcal{S}^{\prime}_{\bot}(\mathbf{R}^3,\mathbf{C}^{K \times 3}),d\nu^g)$ induced by the
gauge transform as
\begin{equation}
U\Psi(\mathbf{A}):=\Psi(\mathbf{A}^f).
\end{equation}
Then, $U$ is a unitary operator in $L^2(\mathcal{S}^{\prime}_{\bot}(\mathbf{R}^3,\mathbf{C}^{K \times 3}),d\nu^{g})$ and
$H^g$ and $H^{g;f}$ are unitary equivalent:
\begin{equation}
H^{g;f}=UH^gU^{-1}.
\end{equation}

\end{theorem}
\begin{proof}
First, we remark that $U$ maps $L^2(\mathcal{S}^{\prime}_{\bot}(\mathbf{R}^3,\mathbf{C}^{K \times 3}),d\nu^g)$ onto itself,
because it preserves the Coulomb gauge. Next, we prove that $U$ is unitary.
For all $\Psi,\Phi\in L^2(\mathcal{S}^{\prime}_{\bot}(\mathbf{R}^3,\mathbf{C}^{K \times 3}),d\nu^g)$
\begin{equation}
\begin{split}
(U\Psi,U\Phi)^g&=\int_{\mathcal{S}^{\prime}_{\bot}(\mathbf{R}^3,\mathbf{R})}\Psi(\mathbf{A}^f)\Phi(\mathbf{A}^f)d\nu^g(\mathbf{A})=\\
&=\int_{\mathcal{S}^{\prime}_{\bot}(\mathbf{R}^3,\mathbf{R})}\Psi(\mathbf{A})\Phi(\mathbf{A})\underbrace{\left|\frac{\partial\mathbf{A}^f}{\partial\mathbf{A}}\right|^{-1}}_{=1}d\nu^g(\mathbf{A})=\\
&=(\Psi,\Phi)^g
\end{split}
\end{equation}
The change of variable is given by equation (\ref{A-trans}) which is affine in $\mathbf{A}$ because the adjoint representation is linear in $A$, which
means $(\text{ad}_g)_*=\text{ad}_g$ for all $g\in G$. Moreover, since
\begin{equation}
\text{ad}_g(A)=(L_g)_*A(R_{g^{-1}})_*,
\end{equation}
the Jacobi determinant reads
\begin{equation}
\det\left((\text{ad}_g)_*\right)=\det\left(\text{ad}_g\right)=\det\left((L_g)_*1_{\mathbf{R}^{3K}}(R_{g^{-1}})_*\right)=\det\left(1_{\mathbf{R}^{3K}}\right)=1.
\end{equation}
Note that the change of variable respects the fibre of the vector bundle $V$, and hence the change of variable formula for the integral is the one of finite dimensional analysis.\\
Next, we prove the unitary equivalence of the Hamilton operators before and after the gauge transform. Their definitions read
\begin{equation}
\begin{split}
H^{g}&=H\left(\mathbf{A},\frac{1}{\imath}\frac{\delta}{\delta \mathbf{A}}\right)\\
H^{g ;f}&=H\left(\mathbf{A}^f,\frac{1}{\imath}\frac{\delta}{\delta \mathbf{A}^f}\right)\\
\end{split}
\end{equation}
For appropriate  $\Psi,\Phi\in L^2(\mathcal{S}^{\prime}_{\bot}(\mathbf{R}^3,\mathbf{C}^{K \times 3}),d\nu^g)$ we have
\begin{equation}
\begin{split}
&(H^{g ;f}\Psi,\Phi)=\\
&=\int_{\mathcal{S}^{\prime}_{\bot}(\mathbf{R}^3,\mathbf{R})}H\left(\mathbf{A}^f,\frac{1}{\imath}\frac{\delta}{\delta \mathbf{A}^f}\right)\Psi(\mathbf{A})\Phi(\mathbf{A})d\nu^g(\mathbf{A})=\\
&=\int_{\mathcal{S}^{\prime}_{\bot}(\mathbf{R}^3,\mathbf{R})}H\left(\mathbf{A},\frac{1}{\imath}\frac{\delta}{\delta \mathbf{A}}\right)U^{-1}\Psi(\mathbf{A})U^{-1}\Phi(\mathbf{A})\underbrace{\left|\frac{\partial\mathbf{A}^f}{\partial\mathbf{A}}\right|^{-1}}_{=1}d\nu^g(\mathbf{A})=\\
&=(H^{g}U^{-1}\Psi,U^{-1}\Phi),
\end{split}
\end{equation}
leading to
\begin{equation}
H^{g;f}=UH^{g}U^{-1}
\end{equation}
on the corresponding domains.\\
\end{proof}
We can therefore conclude that the spectrum of the Hamilton operator for the quantized Yang-Mills problem is gauge invariant.
\subsection{Mass gap}
\begin{conjecture}\label{corspec2}
For any bare coupling constant $g\in[0,g_0[$ the spectrum of the Hamiltonian  $H^g$ contains $0$ as a simple eigenvalue for the vacuum eigenstate, and satisfies
\begin{equation}
\spec(H^g)\subset\{0\}\cup[\eta^g,+\infty[, \text{ for a }\eta^g>0,
\end{equation}
and $\eta^g=O(g^{2(n+1)})$ for any $n\in\mathbf{N}_0$. In particular, there is a mass gap only if $g>0$, and the group $G$ must be non-abelian.
\end{conjecture}

\noindent A proof of Conjecture \ref{corspec2} can be found in \cite{Fa24}. Conjectures \ref{cYMM} and \ref{corspec2} are a formalization of Conjecture \ref{CMI}.

\section{Quark Confinement}

\subsection{Wilson Loop}
\begin{defi} Let $\gamma$ be a closed loop in the Minkowski space, $\mathcal{P}$ the path-ordering operator, $\exp$ the exponential map from the Lie algebra $\mathfrak{g}$ to its Lie group $G$, $\rho:G\rightarrow\text{GL}(\mathbf{C}^{K})$ a faithful representation of $G$ and $\chi:=\tr\circ\rho$ a faithful character of $G$, where $\tr$ is he trace on $\mathcal{L}(\mathbf{C}^{K})$. The Wilson loop is defined for any $\mathbf{A}\in\mathcal{S}_{\bot}^{\prime}(\mathbf{R}^4,\mathbf{C}^{K \times 3})$ as
\begin{equation}
W[\gamma](\mathbf{A}):=\lim_{\varepsilon\rightarrow0+}\chi\left[\mathcal{P}\exp\left(\imath\oint_{\gamma}A_{\alpha}(J_{\varepsilon}*1_{\gamma})dx^{\alpha}\right)\right],
\end{equation}
where $1_{\gamma}$ is the indicator function for the set of points on the curve $\gamma$, $*$ the convolution and
$J_{\varepsilon}\in \mathcal{S}(\mathbf{R}^4,\mathbf{R}^1)$ a mollifier such that $\mathcal{S}^{\prime}-\lim_{\varepsilon\rightarrow+0}J_{\varepsilon}=\delta$.
\end{defi}
\noindent It is straightforward to prove that
\begin{proposition}
For regular fields $\mathbf{A}=\mathbf{a}(t,x)\in L^2_{\text{loc}}(\mathbf{R}^4,\mathbf{C}^{K\times 3})$ the Wilson loop reads
\begin{equation}
W[\gamma](\mathbf{A}):=\chi\left[\mathcal{P}\exp\left(\imath\oint_{\gamma}a_{\alpha}(x^{\alpha}(s))\frac{dx^{\alpha}}{ds}ds\right)\right].
\end{equation}

\end{proposition}
\begin{proposition}
The Wilson loop is gauge invariant, i.e.
\begin{equation}
W[\gamma^f](\mathbf{A}^f) = W[\gamma](\mathbf{A})
\end{equation}
for all gauge transforms $f$ and all $\mathbf{A}\in\mathcal{S}_{\bot}^{\prime}(\mathbf{R}^4,\mathbf{C}^{K \times 3})$.
\end{proposition}
\begin{proof}
It suffices to show it for regular fields.
\begin{equation}
\begin{split}
W[\gamma^f](\mathbf{A}^f)&=\chi\left[\mathcal{P}\exp\left(\imath\oint_{\gamma}a^f_{\alpha}(x^{f;\alpha}(s))\frac{dx^{f;\alpha}}{ds}ds\right)\right]=\\
&=\chi\left[\mathcal{P}\exp\left(\imath\oint_{\gamma}a_{\alpha}(x^{\alpha}(s))\frac{dx^{\alpha}}{ds}ds\right)\right]=W[\gamma](\mathbf{A}).
\end{split}
\end{equation}
\end{proof}

\subsection{Generalization of Elitzur's Theorem}
 We  prove an extension of Elitzur theorem which has been known to hold so far to hold for Quantum Yang-Mills theories on the lattice. For the rationale see \cite{Ma23}.
\begin{theorem}\label{Eli}
Let us consider the quantum Yang-Mills theory introduced in Subsection \ref{quant}.
The only operators that can have non-vanishing expectation values are those which are invariant under local gauge transformation. More exactly,
let $O=O(\mathbf{A})$ be an operator on $L^2(\mathcal{S}^{\prime}_{\bot}(\mathbf{R}^3,\mathbf{C}^{K \times 3}),d\nu^g)$ and $f$ a gauge transform. Then, the operator $O$ and its gauge transform $O^f=O^f(\mathbf{A}):=O(\mathbf{A}^f)$ are unitary equivalent, because
\begin{equation}\label{eq52}
O^{f}=UOU^{-1},
\end{equation}
where $U\Psi(\mathbf{A}):=\Psi(\mathbf{A}^f)$ is a unitary operator induced by the gauge transform $f$.
Moreover, if for all gauge transforms $f$
\begin{equation}\label{eq53}
\left<O^f\right>=\left<O\right>,
\end{equation}
then it must be that
\begin{equation}
\left<O^f\right>=\left<O\right>=0.
\end{equation}
\end{theorem}
\begin{proof}
The proof of equation (\ref{eq52}) is the same as the corresponding proof in Theorem \ref{uequiv}, where we have seen that
 that $U$ maps $L^2(\mathcal{S}^{\prime}_{\bot}(\mathbf{R}^3,\mathbf{C}^{K \times 3}),d\nu^g)$ onto itself,
because it preserves the Coulomb gauge and that $U$ is unitary.
Next, for appropriate  $\Psi,\Phi\in L^2(\mathcal{S}^{\prime}_{\bot}(\mathbf{R}^3,\mathbf{C}^{K \times 3}),d\nu^{g})$ we have
\begin{equation}
\begin{split}
&(O^{f}\Psi,\Phi)^g=
\int_{\mathcal{S}^{\prime}_{\bot}(\mathbf{R}^3,\mathbf{R})}O(\mathbf{A}^f)\Psi(\mathbf{A})\bar{\Phi}(\mathbf{A})d\nu^g(\mathbf{A})=\\
&=\int_{\mathcal{S}^{\prime}_{\bot}(\mathbf{R}^3,\mathbf{R})}O(\mathbf{A})U^{-1}\Psi(\mathbf{A})U^{-1}\bar{\Phi}(\mathbf{A})\underbrace{\left|\frac{\partial\mathbf{A}^f}{\partial\mathbf{A}}\right|^{-1}}_{=1}d\nu^{g}(\mathbf{A})=(OU^{-1}\Psi,U^{-1}\Phi)^g=(UOU^{-1}\Psi,\Phi)^g,
\end{split}
\end{equation}
Finally, for all gauge transform $f$
\begin{equation}
\left<O^f\right>=(O^{f}\Omega_0,\Omega_0)^g=(O\Omega_0,\Omega_0)^g=\left<O\right>,
\end{equation}
which can hold true only if $\left<O^f\right>=\left<O\right>=0$. The proof is completed.\\
\end{proof}

\subsection{Cluster Theorem}

\begin{theorem}[\textbf{Fredenhagen}, \cite{Fr85}]\label{Fred}
Let $\mathcal{H}$ be a Hilbert space, $H$ a selfadjoint operator in $\mathcal{H}$ with $\spec(H)\subset \{0\}\cup[\eta,+\infty[$ for $\eta>0$,
$\Omega_0$ the unique normalized eigenvector of $H$ with eigenvalue $0$,
and let $C_1$, $C_2$ be bounded operators in $\mathcal{H}$ such that
\begin{equation}
[\exp(iHt)C_1\exp(-iHt),C_2]=0
\end{equation}
for $|t|\le T$, for a given $T>0$. Then
\begin{equation}
|(\Omega_0,C_1C_2\Omega_0)-(\Omega_0,C_1\Omega_0)(\Omega_0,C_2\Omega_0)|\le \exp(-\eta T)\left\{\|C_1^*\Omega_0\|\|C_2\Omega_0\|\|C_1\Omega_0\|\|C_2^*\Omega_0\|\right\}^{\frac{1}{2}}.
\end{equation}
\end{theorem}

\subsection{Confinement}
Following \cite{Ch21} we introduce

\begin{defi}[\textbf{Confinement Property}]\label{CP}
A quantum  Yang-Mills Theory satisfies the confinement property if and only if
for any $R>0$
\begin{equation}\label{CPr}
\limsup_{T\rightarrow +\infty}\frac{1}{T}\log |\left<W[\gamma_{R,T}]\right>|\le-V(R)
\end{equation}
for some $V$ such that $V(R)\rightarrow +\infty$ as $R\rightarrow +\infty$, where $\gamma_{R,T}$ is any rectangular loop with side lengths $R$ and $T$.
\end{defi}

\begin{rem}
Why does Property \ref{CP} imply confinement of quarks?
Let us consider a rectangular loop, where the sides of length $R$ represent the lines joining the
quark-antiquark pair at times $0$ and $T$, and the sides of length $T$ represent the trajectories o
f the quark and the antiquark in the time direction.
If $V(R)$ denotes the potential energy of a static quark-antiquark pair separated by distance $R$,
then quantum field theoretic calculations (\cite{Wi74}) indicate that Wilson's loop $W\sim\exp(-V(R)T)$
for $R$ fixed and $T\rightarrow+\infty$. So, if property \ref{CP} holds, then $V(R)$ grows to infinity
as distance $R$ between the quark and the antiquark grows. By the conservation of
energy, this implies that the pair will not be able to separate beyond a certain distance.
\end{rem}

\begin{theorem}\label{thm45}
The Quantum Yang-Mills theory in Conjectures \ref{cYMM} and \ref{corspec2} satisfies for $g\in]0,g_0[$  the confinement property. In particular, the group $G$ must be non-abelian.
\end{theorem}
\noindent In other words for weak coupling constant
\begin{equation}
\boxed{
\text{Conjecture \ref{CMI}: mass gap existence} \Rightarrow \text{Conjecture \ref{QConf}: quark confinement.}
}
\end{equation}
\begin{lemma}\label{l46}
For regular fields under the Coulomb gauge we have
\begin{equation}\label{eqP}
\rho\left(\mathcal{P}\exp\left(\imath\oint_{\gamma_{R,T}}A_{\alpha}dx^{\alpha}\right)\right)=C_1(\mathbf{A}(0,\cdot))C_2(\mathbf{A}(T,\cdot)),
\end{equation}
where
\begin{equation}
C_1(\mathbf{A}):=\rho\left(\exp\left(-\imath R\int_0^1A_1(sR,0,0)ds\right)\right)\qquad
C_2(\mathbf{A}):=\rho\left(\exp\left(+\imath R\int_0^1A_1(sR,0,0)ds\right)\right)
\end{equation}
are bounded multiplication operators on $\mathcal{H}= L^2(\mathcal{S}^{\prime}_{\bot}(\mathbf{R}^3,\mathbf{C}^{K \times 3}), \mathcal{L}(\mathbf{C}^{K}) , d\nu)$ satisfying $C_1C_2=C_2C_1=\mathbb{1}$.\\
For any field under the Coulomb gauge $\mathbf{A}\in\mathcal{S}_{\bot}^{\prime}(\mathbf{R}^3,\mathbf{C}^{K \times 3})$ equation (\ref{eqP}) holds true with
\begin{equation}
\begin{split}
C_1(\mathbf{A})&:=\lim_{\varepsilon\rightarrow0+}\rho\left(\exp\left(-\imath R\int_0^1A_1(\varphi_{\varepsilon}(\cdot-(sR,0,0)))ds\right)\right)\\
C_2(\mathbf{A})&:=\lim_{\varepsilon\rightarrow0+}\rho\left(\exp\left(+\imath R\int_0^1A_1(\varphi_{\varepsilon}(\cdot-(sR,0,0)))ds\right)\right),
\end{split}
\end{equation}
where $(\varphi_{\varepsilon})_{\varepsilon>0}\subset\mathcal{S}_{\bot}(\mathbf{R}^3,\mathbf{C}^{K \times 3})$ such that
$\mathcal{S}^{\prime}-\lim_{\varepsilon\rightarrow0+}\varphi_{\varepsilon}=\delta I^{K\times 3}$.
\end{lemma}
\begin{proof}
It suffices to prove it for regular fields. For arbitrary fields it follows from a density argument. Let us decompose the rectangular loop as sum of its sides:
\begin{equation}
\gamma_{R,T} = \gamma_{R,T}^1 + \gamma_{R,T}^2+ \gamma_{R,T}^2 + \gamma_{R,T}^3+ \gamma_{R,T}^4.
\end{equation}
with the parametrizations
\begin{equation}
\begin{split}
&\gamma_{R,T}^1:[0,1]\rightarrow\mathbf{R}^4, s\mapsto (0,sR,0,0)\\
&\gamma_{R,T}^2:[0,1]\rightarrow\mathbf{R}^4, s\mapsto (sT,R,0,0)\\
&\gamma_{R,T}^3:[0,1]\rightarrow\mathbf{R}^4, s\mapsto (T,(-s+1)R,0,0)\\
&\gamma_{R,T}^4:[0,1]\rightarrow\mathbf{R}^4, s\mapsto ((-s+1)T,0,0,0).
\end{split}
\end{equation}
using the Coulomb gauge and the parametrizations we obtain
\begin{equation}
\begin{split}
\mathcal{P}\exp\left(\imath\oint_{\gamma_{R,T}}A_{\alpha}dx^{\alpha}\right)&=\mathcal{P}\exp\left(\imath\oint_{\gamma_{R,T}}A_{j}dx^{j}\right)=
\mathcal{P}\exp\left(\imath\oint_{\gamma_{R,T}}A_{j}(x^{j}(s))\frac{dx^{j}}{ds}ds\right)=\\
&=\mathcal{P}\exp\left(-\imath R\int_0^1\left[A_{1}(0,sR,0,0)-A_1(T,(-s+1)R,0,0)\right]ds\right)=\\
&=\mathcal{P}\exp\left(-\imath R\int_0^1\left[A_{1}(0,sR,0,0)-A_1(T, sR,0,0)\right]ds\right),
\end{split}
\end{equation}
and hence
\begin{equation}
\rho\left(\mathcal{P}\exp\left(\imath\oint_{\gamma_{R,T}}A_{\alpha}dx^{\alpha}\right)\right)=C_1(\mathbf{A}(0,\cdot))C_2(\mathbf{A}(T,\cdot)).
\end{equation}
\end{proof}

\begin{lemma}\label{l47}
The expectation of the Wilson loop reads
\begin{equation}\label{wtr}
\left<W[\gamma_{R,T}](\mathbf{A})\right>:=\int_{\mathcal{S}_{\bot}^{\prime}(\mathbf{R}^4,\mathbf{C}^{K \times 3})} W[\gamma_{R,T}](\mathbf{A})d\mu^g(\mathbf{A})=\tr\left[\left(\Omega_0,C_1\exp(-TH^g)C_2\Omega_0\right)^g\right].
\end{equation}
\end{lemma}
\begin{proof}
By means of (\ref{repthm2}), see also \cite{Fa24} for details, we can infer
\begin{equation}
\begin{split}
\left(\Omega_0,C_1\exp(-TH^g)C_2\Omega_0\right)^g&=
\int_{\mathcal{S}_{\bot}^{\prime}(\mathbf{R}^4,\mathbf{C}^{K \times 3})}C_1(\mathbf{A}(0,\cdot))C_2(\mathbf{A}(T,\cdot))d\mu^g(\mathbf{A})=\\
&=\int_{\mathcal{S}_{\bot}^{\prime}(\mathbf{R}^3,\mathbf{C}^{K \times 3})}C_1(\mathbf{A})\exp(-TH^g)C_2(\mathbf{A})d\nu^g(\mathbf{A}),
\end{split}
\end{equation}
and hence,
\begin{equation}
\begin{split}
\left<\rho\left(\mathcal{P}\exp\left(\imath\oint_{\gamma_{R,T}}A_{\alpha}dx^{\alpha}\right)\right)\right>&=\int_{\mathcal{S}_{\bot}^{\prime}(\mathbf{R}^4,\mathbf{C}^{K \times 3})}\rho\left(\mathcal{P}\exp\left(\imath\oint_{\gamma_{R,T}}A_{j}dx^{j}\right)\right)d\mu^g(\mathbf{A})=\\
&=\int_{\mathcal{S}_{\bot}^{\prime}(\mathbf{R}^4,\mathbf{C}^{K \times 3})}C_1(\mathbf{A}(0,\cdot))C_2(\mathbf{A}(T,\cdot))d\mu^g(\mathbf{A})=\\
&=\int_{\mathcal{S}_{\bot}^{\prime}(\mathbf{R}^3,\mathbf{C}^{K \times 3})}C_1(\mathbf{A})\exp(-TH^g)C_2(\mathbf{A})d\nu^g(\mathbf{A})=\\
&=\left(\Omega_0,C_1\exp(-TH^g)C_2\Omega_0\right)^g\in\mathcal{L}(\mathbf{C}^{K}).
\end{split}
\end{equation}
Taking the trace on both sides leads to (\ref{wtr}), and the proof is completed.\\
\end{proof}

\begin{proof}[Proof of Theorem \ref{thm45}]
Let us define
\begin{equation}
\tilde{C}_2: = \exp\left(-TH^g\right)C_2.
\end{equation}
We obtain for he commutator:
\begin{equation}
\begin{split}
&[\exp(iH^gt)C_1\exp(-iH^gt),\tilde{C}_2]=\exp(iH^gt)C_1\exp(-iH^gt)\tilde{C}_2-\tilde{C}_2\exp(iH^gt)C_1\exp(-iH^gt)=\\
&=\exp(iH^gt)C_1\exp\left(-(T+\imath t)H^g\right)C_2-\exp\left((-T+\imath t)H^g\right)\underbrace{C_2C_1}_{=\mathbb{1}}\exp\left(-\imath tH^g\right)=\\
&=\exp(iH^gt)\underbrace{C_1C_2}_{=\mathbb{1}}\exp\left(-(T+\imath t)H^g\right)-\exp\left((-T+\imath t)H^g\right)\underbrace{C_2C_1}_{=\mathbb{1}}\exp\left(-\imath tH^g\right)=\\
&=\exp(-TH^g)-\exp(-TH^g)=0,
\end{split}
\end{equation}
because by (\ref{formH}) $H^g$ and $C_2$ commute.\\
By Theorem \ref{Eli} the expectations of $C_1$ and $\tilde{C}_{2}$ vanish:
\begin{equation}
\left<C_1\right>=(\Omega_0,C_1\Omega_0)^g=0\qquad \left<C_2\right>=(\Omega_0,\tilde{C}_2\Omega_0)^g=0.
\end{equation}
By Theorem \ref{corspec2} there exists a positive mass gap $\eta^g>0$ for $g>0$. By Theorem \ref{Fred} we therefore obtain
\begin{equation}\label{ineEta}
|\left(\Omega_0,C_1\exp(-TH^g)C_2\Omega_0\right)^g| =|(\Omega_0,C_1\tilde{C}_2\Omega_0)^g|\le \exp(-\eta^g T)\left\{\|C_1^*\Omega_0\|\|\tilde{C}_2\Omega_0\|\|C_1\Omega_0\|\|\tilde{C}_2^*\Omega_0\|\right\}^{\frac{1}{2}},
\end{equation}
We note that $\|C_1^*\Omega_0\|$ and $\|C_1\Omega_0\|$ do not depend on $T$ and that
\begin{equation}
\begin{split}
&\|\tilde{C}_2\Omega_0\|^2 = (\Omega_0,\tilde{C}_2^*\tilde{C}_2\Omega_0)^g=\int_{\mathcal{S}_{\bot}^{\prime}(\mathbf{R}^4,\mathbf{C}^{K \times 3})}C_2^*(\mathbf{A}(T,\cdot))C_2(\mathbf{A}(T,\cdot))d\mu^g(\mathbf{A})\\
&\|\tilde{C}_2^*\Omega_0\|^2 = (\Omega_0,\tilde{C}_2\tilde{C}_2^*\Omega_0)^g=\int_{\mathcal{S}_{\bot}^{\prime}(\mathbf{R}^4,\mathbf{C}^{K \times 3})}C_2(\mathbf{A}(T,\cdot))C_2^*(\mathbf{A}(T,\cdot))d\mu^g(\mathbf{A}),
\end{split}
\end{equation}
Therefore,
\begin{equation}
\begin{split}
&\|\tilde{C}_2\Omega_0\|^2=\lim_{\varepsilon\rightarrow0+}\int_{\mathcal{S}_{\bot}^{\prime}(\mathbf{R}^4,\mathbf{C}^{K \times 3})}\rho\left(\exp\left(-\imath R\int_0^1(\bar{A}_1-A_1)(T,\varphi_{\varepsilon}(\cdot-(sR,0,0)))ds\right)\right)d\mu^g(\mathbf{A})\\
&\|\tilde{C}_2^*\Omega_0\|^2=\lim_{\varepsilon\rightarrow0+}\int_{\mathcal{S}_{\bot}^{\prime}(\mathbf{R}^4,\mathbf{C}^{K \times 3})}\rho\left(\exp\left(-\imath R\int_0^1(\bar{A}_1-A_1)(T,\varphi_{\varepsilon}(\cdot-(sR,0,0)))ds\right)\right)d\mu^g(\mathbf{A})
\end{split}
\end{equation}
Hence,
\begin{equation}
\begin{split}
&\|\tilde{C}_2\Omega_0\|^2=\|\tilde{C}_2^*\Omega_0\|^2=\\
&=\lim_{\varepsilon\rightarrow0+}\int_{\mathcal{S}_{\bot}^{\prime}(\mathbf{R}^4,\mathbf{C}^{K \times 3})}\rho\left(\exp\left(-2 R\int_0^1\Im(A_1)(T,\varphi_{\varepsilon}(\cdot-(sR,0,0)))ds\right)\right)d\mu^g(\mathbf{A})=\\
&=\lim_{\varepsilon\rightarrow0+}\int_{\mathcal{S}_{\bot}^{\prime}(\mathbf{R}^4,\mathbf{C}^{K \times 3})}\rho\left(\exp\left(+\imath 2R\int_0^1\Im(A_1)(\imath\kappa_{\varepsilon}(\cdot-(T,sR,0,0)))ds\right)\right)d\mu^g(\mathbf{A}),
\end{split}
\end{equation}
where $(\kappa_{\varepsilon})_{\varepsilon>0}\subset\mathcal{S}_{\bot}(\mathbf{R}^4\mathbf{C}^{K \times 3})$ such that
$\mathcal{S}^{\prime}-\lim_{\varepsilon\rightarrow0+}\kappa_{\varepsilon}=\delta I^{K\times 3}$.\par
By the proof the Osterwalder-Schrader regularity axiom (OS1) in Corollary 4.19 in \cite{Fa24}, there is some constant $c>0$ such that
\begin{equation}
\left|\int_{\mathcal{S}_{\bot}^{\prime}(\mathbf{R}^4,\mathbf{C}^{K \times 3})}\rho\left(\exp\left(+\imath \Im(A_1)(\varphi\right)\right)d\mu^g(\mathbf{A})\right|\le \exp\left(c\|\varphi\|_{H^-1(\mathbf{R}^4)}^2\right),
\end{equation}
 for all $\varphi\in\mathcal{S}_{\bot}^{\prime}(\mathbf{R}^4\mathbf{C}^{K \times 3})$.
If $\text{supp}(\varphi)\subset\subset\mathbf{R}^4$, by the Sobolev embedding theorem we have
\begin{equation}
\|\varphi\|_{H^-1(\mathbf{R}^4)}\le\|\varphi\|_{L^1(\mathbf{R}^4)}.
\end{equation}
Therefore,  for all $\varepsilon>0$
\begin{equation}
\left|\int_{\mathcal{S}_{\bot}^{\prime}(\mathbf{R}^4,\mathbf{C}^{K \times 3})}\rho\left(\exp\left(+\imath \Im(A_1)(\mathcal{K}_{\varepsilon}^{T,R}\right)\right)d\mu^g(\mathbf{A})\right|\le \exp\left(c\|\mathcal{K}_{\varepsilon}^{T,R}\|_{L^1(\mathbf{R}^4)}^2\right),
\end{equation}
where
\begin{equation}
\mathcal{K}_{\varepsilon}^{T,R}:=\imath2R\int_0^1(\kappa_{\varepsilon}(\cdot-(T,sR,0,0)))ds.
\end{equation}
Note that for all $\varepsilon>0$ we have $\mathcal{K}_{\varepsilon}^{T,R}\subset B\subset\mathbf{R}^4$ for a compact $B$.\\
We have
\begin{equation}
\begin{split}
&\left|\mathcal{K}_{\varepsilon}^{T,R}(t,x)\right|=2R\int_0^1\kappa_{\varepsilon}(t-T,x_1-sR,x_2,x_3)ds\ge0\\
&\int_{\mathbf{R}^4}\left|\mathcal{K}_{\varepsilon}^{T,R}(t,x)\right|dtdx^3=2R\int_0^1\int_{\mathbf{R}^4}dtdx^3\kappa_{\varepsilon}(t-T,x_1-sR,x_2,x_3)ds\rightarrow 2R\quad(\varepsilon\rightarrow0^+),
\end{split}
\end{equation}
and we conclude that
\begin{equation}
\|\tilde{C}_2\Omega_0\|^2=\|\tilde{C}_2^*\Omega_0\|^2\le\exp(4cR^2),
\end{equation}
and, analogously,
\begin{equation}
\|\tilde{C}_1\Omega_0\|^2=\|\tilde{C}_2^*\Omega_0\|^2\le\exp(4cR^2).
\end{equation}
Inequality (\ref{ineEta}) becomes
\begin{equation}
|\left(\Omega_0,C_1\exp(-TH^g)C_2\Omega_0\right)^g| \le \exp(-\eta^g T)\exp(4cR^2).
\end{equation}
Let us consider a $\beta\in]0,+\infty[$. By solving the quadratic inequality
\begin{equation}\label{ineqR}
-\eta^gT+4cR^2<-\beta TR
\end{equation}
with respect to $R$ we see that
\begin{equation}
0<R<\frac{-\beta T+\sqrt{\beta^2T^2+16C\eta^gT}}{8c}
\end{equation}
implies (\ref{ineqR}). Hence, we obtain
\begin{equation}
|\left(\Omega_0,C_1\exp(-TH^g)C_2\Omega_0\right)^g| \le \exp(-\beta TR),
\end{equation}
and therefore, by Lemma \ref{l47}
\begin{equation}
\left|\left<W[\gamma_{R,T}]\right>\right|=\left|\tr\left[\left(\Omega_0,C_1\exp(-TH^g)C_2\Omega_0\right)^g\right]\right|\le 3K\exp(-\beta TR).
\end{equation}
The confinement property in Definition \ref{CP} is satisfied because
\begin{equation}
\limsup_{T\rightarrow +\infty}\frac{1}{T}\log |\left<W[\gamma_{R,T}]\right>|\le\limsup_{T\rightarrow +\infty}\frac{1}{T}\log(3K\exp(-\beta TR))\le-\beta R,
\end{equation}
and the proof is completed.\\
\end{proof}

\begin{rem} Note that the proof of quark confinement is independent of the choice of the representation $\rho$ of the non-commutative group $G$. This is in line with Seiler's observation in \cite{Se82} page 6 that ``it is curious that the (formal) continuum theory does not have this dependence".
\end{rem}
\section{Conclusion}
For the quantized Yang-Mills $3+1$ dimensional problem we have introduced the Wilson loop, proved an extension of Elitzur's theorem and shown quark confinement for sufficiently small values of the bare coupling constant, provided the existence of the mass gap.

%
%


\end{document}